\documentclass[12pt,reqno]{amsart}
\usepackage{latexsym,amsmath,mathtools,amsfonts,amssymb,amsthm,mathrsfs,color}
\usepackage{enumerate}
\usepackage[T1]{fontenc}
\usepackage{fourier}
\usepackage{xcolor}
\usepackage{fullpage}

\usepackage{tikz}
\usepackage{pgfplots}
\pgfplotsset{compat=1.15}
\usetikzlibrary{arrows}
\usetikzlibrary[patterns]
\definecolor{qqqqff}{rgb}{0,0,1}
\definecolor{ttqqqq}{rgb}{0.2,0,0}
\definecolor{ffqqqq}{rgb}{1,0,0}
\usepackage[latin1]{inputenc}
\usepackage{anysize}
\usepackage{graphicx}
\usepackage{url}
\usepackage[noadjust]{cite}

\usepackage{comment}
\newcommand{\bb}[1]{\mathbb{#1}}
\newcommand{\cc}[1]{\mathcal{#1}}

\theoremstyle{definition}
\newtheorem{theorem}{Theorem}[section]
\newtheorem{question}[theorem]{Question}

\newtheorem{corollary}[theorem]{Corollary}
\newtheorem{lemma}[theorem]{Lemma}
\newtheorem{claim}[theorem]{Claim}

\title{Expansion, long cycles, and complete minors in supercritical random
subgraphs of the hypercube}
	\date{}

	\author{Joshua Erde$^{*}$, Mihyun Kang$^{*}$, and Michael Krivelevich$^{\ddagger}$ \\ \\
		\today}
	\thanks{$^{*}$ 
		Institute of Discrete Mathematics, 
		Graz University of Technology, 
		Steyrergasse 30,
		8010 Graz,
		Austria,  
		{\tt \{erde,kang\}@math.tugraz.at}.
		Supported by Austrian Science Fund (FWF): I3747\phantom{}}
	\thanks{$^{\ddagger}$ 
		School of Mathematical Sciences, 
		Sackler Faculty of Exact Sciences, 
		Tel Aviv University,  
		Tel Aviv 6997801,
		Israel, 
		{\tt  krivelev@tauex.tau.ac.il}.
		Supported in part by USA-Israel BSF grant 2018267, and by ISF grant 1261/17.}
\begin{document}
	
	\begin{abstract}
Analogous to the case of the binomial random graph $G(d+1,p)$, it is known that the behaviour of a random subgraph of a $d$-dimensional hypercube, where we include each edge independently with probability $p$, which we denote by $Q^d_p$, undergoes a phase transition when $p$ is around $\frac{1}{d}$. More precisely, standard arguments show that significantly below this value of $p$, with probability tending to one as $d \to \infty$ (whp for short) all components of this graph have order $O(d)$, whereas Ajtai, Koml\'{o}s and Szemer\'{e}di \cite{AKS81} showed that significantly above this value, in the \emph{supercritical regime}, whp there is a unique `giant' component of order $\Theta\left(2^d\right)$. 
In $G(d+1,p)$ much more is known about the complex structure of the random graph which emerges in this supercritical regime. For example, it is known that in this regime whp $G(d+1,p)$ contains paths and cycles of length $\Omega(d)$, as well as complete minors of order $\Omega\left(\sqrt{d}\right)$. In this paper we obtain analogous results in $Q^d_p$. In particular, we show that for supercritical $p$, i.e., when $p=\frac{1+\epsilon}{d}$ for a positive constant $\epsilon$, whp $Q^d_p$ contains a cycle of length $\Omega\left(\frac{2^d}{d^3(\log d)^3} \right)$ and a complete minor of order $\Omega\left(\frac{2^{\frac{d}{2}}}{d^3(\log d)^3 }\right)$. In order to prove these results, we show that whp the largest component of $Q^d_p$ has good edge-expansion properties, a result of independent interest. We also consider the genus of $Q^d_p$ and show that, in this regime of $p$, whp the genus is $\Omega\left(2^d\right)$.
	\end{abstract}

	\maketitle

	\section{Introduction}
The binomial random graph model $G(n,p)$, introduced by Gilbert \cite{G59}, is a random variable on the subgraphs of the complete graph $K_n$ whose distribution is given by including each edge in the subgraph independently with probability $p$. Since its introduction, this model has been extensively studied. A particularly striking feature of this model is the `phase transition' that it undergoes at $p = \frac{1}{n}$, exhibiting vastly different behaviour when $p = \frac{1 - \varepsilon}{n}$ to when $p = \frac{1+\varepsilon}{n}$ (where $\varepsilon$ is a positive constant). For more background on the theory of random graphs, see \cite{B01,JLR00,FK16}.
	
More recently, the following generalisation of this model has been the object of study: Suppose $G$ is an arbitrary graph with minimum degree $\delta(G)$ at least $d$ and let $G_p$ denote the random subgraph of $G$ obtained by retaining each edge of $G$ independently with probability $p$. When $G= K_{d+1}$, the complete graph on $d+1$ vertices, we recover the model $G(d+1,p)$. 
	
If we take the probability $p$ to be a function of (the lower bound $d$ for) the minimum degree of $G$, rather than the order (i.e., number of vertices) of $G$, it has been shown that the model $G_p$ shares many properties with $G(d+1,p)$. In particular, some of the complex behaviour which occurs whp\footnote{Throughout the paper, all asymptotics will be considered as $d \to \infty$ and so, in particular, whp (with high probability) means with probability tending to one as $d \to \infty$} in $G(d+1,p)$ once we pass the critical point of $p=\frac{1}{d}$ also occurs whp in $G_p$ in the same regime of $p$.
For example, when $p= \frac{1+\varepsilon}{d}$ for $\epsilon > 0$, it has been shown that whp $G_p$ contains a path or cycle of length linear in $d$, see \cite{KS13,KS14,EJ18}. Furthermore, for this range of $p$ it has been shown by Frieze and Krivelevich \cite{FK13} that whp $G_p$ is non-planar and by Erde, Kang and Krivelevich \cite{EKK20} that in fact, whp $G_p$ contains a complete minor of order\footnote{The notation $\tilde{\Omega}\left(\cdot\right)$ here is hiding a polylogarithmic factor in $d$.} $\tilde{\Omega}\left(\sqrt{d}\right)$ and hence has genus $\tilde{\Omega}(d)$. Similarly, when $p = \omega\left(\frac{1}{d}\right)$, whp $G_p$ contains a path or cycle of length $(1-o(1))d$, as shown in \cite{R14,KLS15}. In addition, when $p = (1+\varepsilon)\frac{\log d}{d}$ for $\epsilon >0$, whp $G_p$ contains a path of length $d$ and in fact even a cycle of length $d+1$, see \cite{KLS15,GNS17} respectively. All of these results generalise known results about the binomial model $G(d+1,p)$.

The model $G(d+1,p)$ shows that these results are optimal when $G$ can be an \emph{arbitrary} graph, but for \emph{specific} graphs $G$, they may be far from the truth. One particular graph for which the model $G_p$ has been studied extensively is the \emph{hypercube} $Q^d$, the graph with vertex set $V\left(Q^d\right) = \{0,1\}^d$ and in which two vertices are adjacent if they differ in exactly one coordinate. Throughout this paper we will write $n := 2^d$ for the order of the hypercube and we note that $|E\left(Q^d\right)| = \frac{nd}{2}$.

Random subgraphs of the hypercube were first studied by Saposhenko \cite{S67} and Burtin \cite{B77}, who showed that $Q^d_p$ has a threshold for connectivity at $p=\frac{1}{2}$; for a fixed constant $p < \frac{1}{2}$ whp $Q^d_p$ is disconnected and for a fixed constant $p > \frac{1}{2}$ whp $Q^d_p$ is connected, and this result was strengthened by Erd\H{o}s and Spencer \cite{ES79} and by Bollob\'{a}s \cite{B83}. Bollob\'{a}s \cite{B90} also showed that $p=\frac{1}{2}$ is the threshold for the existence of a perfect matching in $Q^d_p$. Very recently, answering a longstanding open problem, Condon, Espuny D{\'\i}az, Girao, K{\"u}hn and Osthus \cite{CDGKO21} showed that $p=\frac{1}{2}$ is also the threshold for the existence of a Hamilton cycle in $Q^d_p$.

It was conjectured by Erd\H{o}s and Spencer \cite{ES79} that $Q^d_p$ should undergo a similar phase transition at $p=\frac{1}{d}$ as $G(d+1,p)$ does: It is relatively easy to see, by a coupling with a branching process, that when $p = \frac{1-\epsilon}{d}$ for $\epsilon>0$, whp all components of $Q^d_p$ have order $O(d)$, but they conjectured that when $p= \frac{1+\epsilon}{d}$, whp $Q^d_p$ contains a unique `giant' component whose order is linear in $n=2^d$. This conjecture was confirmed by Ajtai, Koml\'{o}s and Szemer\'{e}di \cite{AKS81}.

\begin{theorem}[\cite{AKS81}]\label{t:AKS}
Let $\epsilon > 0$ be a constant and let $p= \frac{1+\epsilon}{d}$. Then there exists a $\gamma >0$ such that whp the largest component of $Q^d_p$ has order at least $\gamma n$.
\end{theorem} 

Ajtai, Koml\'{o}s and Szemer\'{e}di also indicated that it is possible to show that there is a unique `giant' component of linear order in $n$, whose order is $\left(\gamma(\epsilon) +o(1)\right) n$ where $\gamma(\epsilon)$ is the survival probability of the Po$(1+\epsilon)$ branching process. These results were later extended to a wider range of $p$, describing more precisely the component structure of $Q^d_p$ when $p = \frac{1+\epsilon}{d}$ with $\epsilon=o(1)$ by Bollob\'{a}s, Kohayakawa and {\L}uczak \cite{BKL92}, by Borgs, Chayes, van der Hofstad, Slade and Spencer \cite{BCVSS06}, and by Hulshof and Nachmias \cite{HN20}, with the correct width of the critical window in this model being only recently identified by van der Hofstad and Nachmias \cite{HN17}. For a more detailed background on the phase transition in this model, see the survey of van der Hofstad and Nachmias \cite{HN14}.

If we consider the structure of $Q^d_p$ above this critical threshold of $p=\frac{1}{d}$, then we can conclude from the results in the $G_p$ model mentioned earlier that whp $Q^d_p$ contains a path and cycle of length $\Omega(d)$, has genus $\tilde{\Omega}(d)$ and contains a complete minor of order $\tilde{\Omega}(\sqrt{d})$. However, these results seem far from optimal. Indeed, already above this threshold whp $Q^d_p$ contains a component whose order is linear in $n=2^d$, and so exponentially large in $d$. Hence, it is possible that $Q^d_p$ could contain a path and cycle of length linear in $n$. Similarly, it is known that the genus of $Q^d$ is exponentially large in $d$. In fact, the genus of $Q^d$ has been shown to be precisely $(d-4)2^{d-3}+1$ for $d \geq 2$, see \cite{R55,BH65}, and the size of the largest complete minor in $Q^d$ was shown to be between $2^{\lfloor\frac{d-1}{2}\rfloor}$ and $\sqrt{d}2^{\frac{d}{2}} + 1$ in \cite{CS07}. Hence, it is possible that whp the genus of $Q^d_p$ in the supercritical regime is linear in $n$ and whp the size of the largest complete minor in $Q^d_p$ is of order $\Omega\left(\sqrt{n}\right)$.

The main aim of this paper is to show that both of these statements are true, at least up to a polylogarithmic factor in $n$. More precisely, we prove the following theorems concerning the length of the longest cycle and the size of the largest complete minor in $Q^d_p$.

	\begin{theorem}\label{t:longcycle}
	Let $\epsilon > 0$ be a constant and let $p = \frac{1+\epsilon}{d}$. Then whp $Q^d_p$ contains a cycle of length $\Omega \left( \frac{2^d}{d^3(\log d)^3 } \right)$.
	\end{theorem}
	
	\begin{theorem}\label{t:Hadwiger}
	Let $\epsilon > 0$ be a constant and let $p = \frac{1+\epsilon}{d}$. Then whp $Q^d_p$ contains a $K_t$-minor with $t = \Omega \left( \frac{2^{\frac{d}{2}}}{d^3(\log d)^3 } \right)$.
	\end{theorem}
We note that here, and elsewhere in the paper, there may be a dependence on the parameter $\epsilon$ in the implicit constants in the bounds.
	
In fact, both of these theorems follow from a more general statement, of independent interest, which says that whp the largest component of $Q^d_p$ has \emph{good edge-expansion} properties. That is, every large enough subset of the largest component `expands', in the sense that it has a large edge boundary. The notion of graph expansion has turned out to have fundamental importance in diverse areas of discrete mathematics and computer science: For a comprehensive introduction to expander graphs, see the survey of Hoory, Linial and Widgerson \cite{HLW06}. In particular, notions of expansion have turned out to be a powerful tool in the study of random structures: See for example the survey paper of Krivelevich \cite{K19}.

In what follows, given a subset $S \subseteq V\left(Q^d\right)$, we will write $\partial_v(S)$ and $\partial_e(S)$ for the \emph{vertex} and \emph{edge} boundary of $S$ in $Q^d$ respectively. That is, $\partial_v(S)$ is the set of vertices in $V\left(Q^d\right) \setminus S$ which have a neighbour in $S$ and $\partial_e(S)$ is the set of edges between $S$ and $V\left(Q^d\right)\setminus S$. Furthermore, we will write $\partial_{v,p}(S)$ and $\partial_{e,p}(S)$ for the vertex and edge boundary of $S$ in $Q^d_p$ respectively. Also, given constants $\alpha,\beta >0$ and a statement $A$, we will write `Let $\alpha \ll \beta$. Then $A$ holds' to indicate that there is some fixed, implicit function $f$ such that $A$ holds for all $\alpha \leq f(\beta)$.

\begin{theorem}\label{t:edgeexpansion}
	Let $0 < \alpha \ll \epsilon$ be constants, let $p = \frac{1+\epsilon}{d}$ and let $L_1=L_1\left(Q^d_p\right)$ be the largest component of $Q^d_p$. Then there exists a $\beta >0$ such that whp every subset $S \subseteq L_1$ satisfying ${\alpha n \leq |S| \leq\frac{|V(L_1)|}{2}}$ has an edge boundary satisfying $\left|\partial_{e,p}(S)\right| \geq \beta \frac{n}{d^3 (\log d)^2}$.
	\end{theorem}
	
As an immediate corollary, since $Q^d$ is $d$-regular, and so every set of $m$ edges in $Q^d$ must span at least $\frac{m}{2d}$ vertices, we can deduce that whp the largest component of $Q^d_p$ also has \emph{good vertex-expansion}. In fact, since in $Q^d_p$ we do not expect significantly fewer than $m$ vertices to span $m$ edges, we can do slightly better.
	
	\begin{corollary}\label{c:vertexexpansion}
	Let $0 < \alpha \ll \epsilon$ be constants, let $p = \frac{1+\epsilon}{d}$ and let $L_1=L_1\left(Q^d_p\right)$ be the largest component of $Q^d_p$. Then there exists a $\beta >0$ such that whp every subset $S \subseteq L_1$ satisfying ${\alpha n \leq |S| \leq\frac{|V(L_1)|}{2}}$ has a vertex boundary satisfying $\left|\partial_{v,p}(S)\right| \geq \beta \frac{n}{d^3(\log d)^3 }$.
	\end{corollary}
	
We note that, whilst it is perhaps not immediately clear that we cannot improve Theorem \ref{t:edgeexpansion} to show constant edge-expansion for linear sized subsets of $L_1 = L_1\left(Q^d_p\right)$, it is clear that we cannot hope to show constant vertex-expansion. Indeed, $Q^d$ contains a separator of size $O\left(\frac{n}{\sqrt{d}}\right)$, namely the middle layer, and it is not hard to show that whp this separator splits $L_1$ roughly in half, and so whp there is a linear sized subset of $L_1$ whose vertex boundary has size $O\left(\frac{n}{\sqrt{d}}\right)$. However, it then follows from a similar argument as in Corollary \ref{c:vertexexpansion} that whp $L_1$ contains some linear sized subset whose edge boundary has size $O\left(\frac{n \log d}{\sqrt{d}}\right)$.

In particular, this implies that we cannot deduce optimal versions of Theorems \ref{t:longcycle} or \ref{t:Hadwiger}, that is, ones without the polynomial factors in $d$, straightforwardly from the expansion properties of the largest component of $Q^d_p$.

In fact, it is perhaps surprising that the expansion properties of $L_1\left(Q^d_p\right)$ are useful at all, since random subgraphs of the hypercube have been used by Moshkowitz and Shapira \cite{MS18} to construct graphs which have in some way the worst possible expansion properties.

Whilst Theorem \ref{t:Hadwiger} already implies that the genus of $Q^d_p$ is $\Omega\left(\frac{n}{d^6 (\log d)^6}\right)$, with a more careful argument we are able to give a better bound. 

\begin{theorem}\label{t:genus}
Let $\epsilon > 0$ be a constant and let $p = \frac{1+\epsilon}{d}$. Then whp the genus of $Q^d_p$ is $\Omega(n)$.
\end{theorem}

Note that, since whp $\left|E\left(Q^d_p\right)\right| = \Theta(n)$, and clearly the genus of a graph is at most the number of edges, Theorem \ref{t:genus} is optimal up to the value of the leading constant.

The paper is structured as follows. In Section \ref{s:prelim} we collect some lemmas which will be useful in the rest of the paper. Then, in Section \ref{s:expansion} we prove our main results on the expansion properties of the largest component in $Q^d_p$ (Theorem \ref{t:edgeexpansion} and Corollary \ref{c:vertexexpansion}), and use them to deduce Theorems \ref{t:longcycle} and \ref{t:Hadwiger}. In Section \ref{s:genus} we give a proof of Theorem \ref{t:genus}, and then in Section \ref{s:discussion} we mention some open problems.
	
	\section{Preliminaries}\label{s:prelim}
For real numbers $x,y,z$ we will write $x=y\pm z$ to mean that $y-z \leq x \leq y+z$.
	
	We will want to use the following simple lemma, which is a slight adaptation of a result in \cite{KN06}, to decompose a tree into roughly equal sized parts.
	\begin{lemma}\label{l:treedecomp}
	Let $T$ be a tree such that $\Delta(T) \leq C_1$, all but $r$ vertices of $T$ have degree at most $C_2\leq C_1$ and $|V(T)| \geq \ell$ for some $C_1,C_2,\ell,r >0$. Then there exist disjoint vertex sets $A_1,\ldots, A_s \subseteq V(T)$ such that
		\begin{itemize}
			\item $V(T) = \bigcup_{i=1}^s A_i$; 
			\item $T[A_i]$ is connected for each $1 \le i \le s$;
			\item $\ell \leq |A_i| \leq  C_1 \ell$ for each $1 \le i \le r$; and
			\item $\ell \leq |A_i| \leq  C_2 \ell$ for each $r < i \le s$.
		\end{itemize}
	\end{lemma}
	\begin{proof}
We choose an arbitrary root $w$ for $T$. For a vertex $v$ in a rooted tree $S$, let us write $S_v$ for the subtree of $S$ rooted at $v$. 

We construct the vertex sets $A_i$ inductively. Let us start by setting $T(0)=T$. Given a tree $T(i)$ rooted at $w$ such that $|V\left(T(i)\right)| \geq \ell$, let $v_i$ be a vertex of maximal distance from $w$ such that $\left|V\left(T(i)_{v_i}\right)\right| \geq \ell$. We take $A_{i+1} = V\left(T(i)_{v_i}\right)$ and let $T(i+1) = T(i) \setminus T(i)_{v_i}$. We stop when $|V\left(T(i)\right)| < \ell$, and in that case we add $V\left(T(i)\right)$ to the final $A_i$. Finally, let us re-order the sets $A_i$ so that they are non-increasing in size.

We claim that the sets $A_1, A_2, \ldots, A_s$ satisfy the conclusion of the lemma. Indeed, the first two properties are clear by construction. Also, we note that by our choice of $v_i$, $|T(i)_x| < \ell$ for every child $x$ of $v_i$. Hence, if $v_i \neq w$, then $v_i$ has $d(v_i)-1$ children and so, since $A_i$ is the union of the vertices of $T(i)_x$ over all children $x$ of $v_i$, together with a set of size less than $\ell$, it follows that $|A_{i+1}| \leq d(v_i) \ell$. Furthermore, if $v_i=w$, then we note that $A_{i+1}=V(T(i))$ and so $|A_{i+1}| \leq d(w)(\ell-1) + 1 \leq d(w) \ell$. Therefore, since all but $r$ vertices of $T$ have degree at most $C_2$, the third and fourth properties also hold.
	\end{proof}

We will also need the following results, which allow us to deduce the existence of a long cycle and a large complete minor from vertex-expansion properties of a graph. For wider context on properties of expanding graphs, see the survey of Krivelevich \cite{K19}.

\begin{theorem}[{\cite[Theorem 1]{K19a}}]\label{t:cycleexpander}
Let $k \geq 1, t \geq 2$ be integers. Let $G$ be a graph on more than $k$ vertices satisfying
\[
\left|\partial_v (W)\right| \geq t, \qquad \text{for every } W \subset V \text{ with } \frac{k}{2} \leq |W| \leq k.
\]
Then $G$ contains a cycle of length at least $t+1$.
\end{theorem}

The following theorem allows us to deduce the existence of a large complete minor in a graph without any small separators. It is easy to see that graphs with good vertex-expansion do not contain any small separators, and in fact it is known (see \cite[Section 5]{K19}) that the converse is true, in the sense that graphs without small separators must contain large induced subgraphs with good vertex-expansion.

\begin{theorem}[{\cite[Theorem 1.2]{KR10}}]\label{t:minorexpander}
Let $G$ be a graph with $N$ vertices and with no $K_t$-minor. Then $G$ contains a subset $S$ of order $O\left(t \sqrt{N}\right)$ such that each connected component of $G \setminus S$ has at most $\frac{2}{3}N$ vertices.
\end{theorem}

We will use the following Chernoff type bounds on the tail probabilities of the binomial distribution, see e.g. \cite[Appendix A]{AS}.
\begin{lemma}\label{l:Chernoff}
Let $N \in \mathbb{N}$, let $p \in [0,1]$ and let $X \sim \text{Bin}(N,p)$.

\begin{enumerate}[(i)]
\item\label{i:chernoff1} For every positive $a$ with $a \leq \frac{Np}{2}$,
\[
\mathbb{P}\left(\left|X -Np \right| > a\right) < 2 \exp\left(-\frac{a^2}{4Np} \right).
\]
\item\label{i:chernoff2} For every positive $b$,
\[
\mathbb{P}\left(X > bNp \right) \leq \left(\frac{e}{b}\right)^{b Np}.
\]
\end{enumerate}
\end{lemma}

In particular, the following consequence of Lemma \ref{l:Chernoff} on the edge boundary of large sets in $Q^d_p$ will be useful.
\begin{lemma}\label{l:edgeboundary}
Let $k,c,\epsilon>0$ be constants and let $p=\frac{1+\epsilon}{d}$. Then whp every subset $X \subseteq V\left(Q^d\right)$ of size $|X| \geq c n d^{-k}$ has edge boundary in $Q^d_p$ such that
\[
\left|\partial_{e,p}(X) \right| \leq |X| \log d.
\]
\end{lemma}
\begin{proof}
Since any subset $W \subseteq V\left(Q^d\right)$ satisfies $\left|\partial_e(W)\right| \leq d|W|$, it follows from Lemma \ref{l:Chernoff} \eqref{i:chernoff2} that
\begin{align*}
\mathbb{P}\left(\left|\partial_{e,p}(W)\right| \geq |W| \log d\right) &\leq \mathbb{P}\left(\text{Bin}(d |W|,p) \geq |W|\log d\right)\\
&\leq \left(\frac{e(1+\epsilon)}{\log d}\right)^{|W|\log d} \leq d^{-(k+1)|W|}.
\end{align*}
Hence, the probability that there exists a subset $X\subseteq V\left(Q^d_p\right)$ not satisfying the conclusion of the lemma is at most
\[
\sum_{i \geq c n d^{-k}} \binom{n}{i} d^{-(k+1)i} \leq \sum_{i \geq c n d^{-k}}   \left( \frac{en}{d^{k+1}i}\right)^i \leq \sum_{i \geq c n d^{-k}}\left( \frac{e}{cd}\right)^i = o(1).
\]

\end{proof}
We will also need to use the following correlation inequality, which is a consequence of an inequality of Harris \cite{H60} and is itself a special case of the FKG-inequality: See for example \cite[Section, 6]{AS}. 
 
\begin{lemma}\label{l:Harris}
Let $\mathcal{A}$ and $\mathcal{B}$ be families of subgraphs of $Q^d$ which are closed under taking supergraphs. Then 
\[
\mathbb{P}\left(Q^d_p \in \mathcal{A} \cap \mathcal{B}\right) \geq \mathbb{P}\left(Q^d_p \in \mathcal{A} \right)\mathbb{P}\left(Q^d_p \in \mathcal{B}\right).
\]
\end{lemma}

Finally, we will use the following result on vertex-isoperimetry in the hypercube, which is a useful strengthening of Harper's \cite{H66} well-known isoperimetric inequality.

\begin{theorem}[{\cite[Corollary 2]{BL97}}]\label{t:vtxiso}
Let $A \subseteq V\left(Q^d\right)$ and let $0 \leq \lambda < 1$ and $k \in \mathbb{N}$ be such that 
\[
|A| = \sum_{i=0}^k \binom{d}{i} + \lambda \binom{d}{k+1} \leq \frac{n}{2}.
\]
Then there is a matching from $A$ to its complement $A^c$ of size at least $(1-\lambda)\binom{d}{k} + \lambda \binom{d}{k+1}$.

In particular, for any subset $A \subset V\left(Q^d\right)$ such that both $A$ and $A^c$ have size $\Omega(n)$, there is a matching from $A$ to $A^c$ of size $\Omega\left(\frac{n}{\sqrt{d}}\right)$.
\end{theorem}

\section{Edge-expansion in the largest component}\label{s:expansion}
We will need some properties of the largest component of $Q^d_p$ given in \cite{AKS81}, where the statement below is claimed, although the details of the proof are not spelled out. For completeness we include a proof in Appendix \ref{a:AKS}.
\begin{theorem}[\cite{AKS81}]\label{t:AKSfine}
Let $0 < \delta_2 \ll \delta_1$, let $p=\frac{1+\delta_1}{d}$ and let $\gamma_1$ be the survival probability of the ${\text{Po}(1+\delta_1)}$ branching process. Then there exist constants $\delta_3,\delta_4,\delta_5>0$, which can be chosen such that $\delta_3 \ll \delta_2$, such that, if $Y$ is the set of vertices of $Q^{d}_p$ contained in components of order at least $\delta_3 d^2$, then whp:
\begin{enumerate}[(i)]
\item\label{i:largecomponent} there is a unique component $L_1$ of order at least $\delta_2 n$ in $Q^d_p$ and $|V(L_1)| = (\gamma_1 \pm \delta_2) n$;
\item\label{i:largevtcs} $| V(L_1) \triangle Y| \leq \delta_2 n$;
\item\label{i:badvtcs} the set of vertices 
\[
X = \left\{ x \in V\left(Q^d\right) \colon \left|\partial_v(x) \cap Y \right| \leq \delta_4 d \right\}
\]
satisfies $|X| \leq 2^{(1-\delta_5)d}$.
\end{enumerate}
\end{theorem}

With this theorem in hand, let us briefly sketch the strategy to prove Theorem \ref{t:edgeexpansion} on the edge-expansion properties of the largest component of $Q^d_p$. We will use a sprinkling argument, viewing $Q^d_p$ as the union of two independent random subgraphs $Q^d_{q_1}$ and $Q^d_{q_2}$ where $q_1$ is a supercritical probability chosen sufficiently close to $p$ such that, if we denote by $L_1'$ and $L_1$ the largest component in $Q^d_{q_1}$ and $Q^d_{p}$ respectively, then whp $L_1'$ contains most of the vertices in $L_1$, which we can guarantee by Theorem \ref{t:AKSfine}.

After exposing $Q^d_{q_1}$, it will follow from Property \eqref{i:badvtcs} of Theorem \ref{t:AKSfine} that whp most of the vertices contained in components of $Q^d_{q_1}$ of order $\Omega\left(d^2\right)$ are in $L'_1$. Using Lemma \ref{l:treedecomp}, we can split the subgraph of $Q^d_{q_1}$ consisting of the union of these components into a collection $\mathcal{C}$ of connected pieces which all have polynomial size in $d$, specifically almost all the pieces will have size around $d^{\frac{3}{2}}$. We then show that whp for any partition $\cc{C} = \cc{C}_A \cup \cc{C}_B$ of $\cc{C}$ into two parts, each covering a positive proportion of the vertices in $L'_1$, there are exponentially many edge-disjoint paths between vertices in the two partition classes in $Q^d_{q_2}$. We do so via a union bound, showing that the probability that any partition fails to have this property is much smaller than the total number of possible partitions. For partitions in which the two partition classes have many shared neighbours we will use property \eqref{i:badvtcs} of Theorem \ref{t:AKSfine} to find these paths and otherwise we will additionally use Theorem \ref{t:vtxiso}.

Given this, we note that any large subset $S \subseteq V(L_1)$ must either split apart exponentially many of the pieces in $\cc{C}$, in which case for each piece which is split there will be some edge of $Q^d_{q_1}$ present in $\partial_{e,p}(S)$, or there must be a partition $\cc{C} = \cc{C}_A \cup \cc{C}_B$ of $\cc{C}$ into two parts, each covering a positive proportion of $L'_1$, such that each piece in $\cc{C}_A$ is contained in $S$ and almost every piece in $\cc{C}_B$ is disjoint from $S$. Our claim then implies that there are many edge disjoint paths in $Q^d_{q_2}$ between the two partition classes, almost all of which must contribute at least one edge to $\partial_{e,p}(S)$.

\begin{proof}[Proof of Theorem \ref{t:edgeexpansion}]
During the proof we will introduce a series of constants $c_1,c_2,\ldots,$ with the understanding that each $c_i$ will be as small as is necessary for the argument in terms of $\alpha, \epsilon$ and all $c_j$ with $j < i$.

Let us fix $c_2 \ll c_1$, let $\gamma$ be the survival probability of a Po$(1+\epsilon)$ branching process and let us choose $\delta_1 < \epsilon$ such that $\gamma_1$, the survival probability of a Po$(1+\delta_1)$ branching process, is such that $\gamma-\gamma_1 \leq c_2$. We take $\delta_1$ and $\delta_2=c_1$ in Theorem \ref{t:AKSfine} to find $c_2 \gg c_3=\delta_3$, $c_4 = \delta_4$ and $c_5 = \delta_5$. Note that we may assume that $c_5 \ll c_4 \ll c_3$.

Our plan is to expose $Q^d_p$ in two steps, with $q_1= \frac{1+\delta_1}{d}$ and $q_2 = \frac{p-q_1}{1-q_1}$, where we may assume that $q_2 \geq \frac{c_6}{d}$. Let us write $Q_1 = Q^d_{q_1}$ and $Q_2 = Q_1 \cup Q^d_{q_2}$ where the two random subgraphs are chosen independently, so that $Q_2$ has the same distribution as $Q^d_p$.

By Theorem \ref{t:AKSfine}, whp the largest component $L'_1$ in $Q_1$ is such that $|V(L'_1)| = (\gamma_1 \pm c_1) n = (\gamma \pm 2c_1)n$. Furthermore, if $Y$ is the set of vertices of $Q_1$ contained in components of order at least $c_3 d^2$ then whp $|V(L'_1) \triangle Y| \leq c_1 n$ and
\begin{equation}\label{e:X}
X = \left\{ x \in V\left(Q^d\right) \colon |\partial_v(x) \cap Y| \leq c_4 d \right\}
\end{equation}
satisfies $|X| \leq 2^{(1-c_5)d}$.

Recall that $L_1$ is the largest component of $Q^d_p$. We first show that whp the symmetric difference of $V(L_1)$ and $Y$ is small.
\begin{claim}\label{c:giant}
Whp $|V(L_1) \triangle Y| \leq 4c_1n$.
\end{claim}
\begin{proof}[Proof of Claim \ref{c:giant}]
We apply Theorem \ref{t:AKSfine} with $\delta_1 = \epsilon$ and $\delta_2 = c_1$ to $Q_2 = Q^d_p$. It follows that whp there is a unique component $L_1$ in $Q_2$ of order at least $c_1 n$ and that $|V(L_1)| = (\gamma \pm c_1) n$. However, since $Q_1 \subseteq Q_2$ and $|V(L_1')| \geq (\gamma - 2c_1) n \geq c_1 n$, it follows that $L'_1 \subseteq L_1$. Hence, ${|V(L_1) \setminus V(L'_1)| \leq 3 c_1 n}$, and so $|V(L_1) \triangle Y| \leq 4 c_1 n$.
\end{proof}

Next, we want to use Lemma \ref{l:treedecomp} to split each component of $Q_1$ of order at least $c_3 d^2$ into connected pieces of roughly equal size. In order to do so we need the following bound on the degree sequence of $Q_1$.
\begin{claim}\label{c:degrees}
Whp $Q_1$ contains at most $nd^{-5}$ vertices of degree at least $\log d$.
\end{claim}
\begin{proof}[Proof of Claim \ref{c:degrees}]
For any fixed vertex $v \in V\left(Q^d\right)$, the degree of $v$ in $Q_1 = Q^d_{q_1}$ is distributed as Bin$(d,q_1)$, and so by Lemma \ref{l:Chernoff} \eqref{i:chernoff2} we have that
\[
\mathbb{P}\left( d_{Q_1}(v) \geq \log d  \right) \leq \left(\frac{e(1+\delta_1)}{ \log d } \right)^{ \log d } \leq d^{-\frac{\log \log d}{2}}.
\]

It follows that the expected number of vertices in $Q_1$ with degree at least $\log d$ is at most $nd^{-\frac{\log \log d}{2}}$. Hence, by Markov's inequality, whp there at most $nd^{-5}$ vertices with degree at least $\log d$.
\end{proof}
We note that the above argument is rather unoptimised and with a little more care, the bound on the degree of the exceptional vertices could be improved from $\log d$ to $\frac{C \log d}{\log \log d}$ for some suitably large constant $C$. However, for ease of presentation we have not attempted to optimise the logarithmic factors in the proof.

Assuming that Claim \ref{c:degrees} holds, by applying Lemma \ref{l:treedecomp} with $\ell = \frac{1}{c_{10}} d^{\frac{3}{2}}$ to a spanning tree of each component of $Q_1$ of order at least $c_3 d^2$, noting that $\Delta\left(Q_1\right) \leq \Delta\left(Q^d\right) = d$, we can split the collection of these components into connected pieces such that at most $nd^{-5}$ of the pieces have size between $\ell = \frac{1}{c_{10}} d^{\frac{3}{2}}$ and $\ell d = \frac{1}{c_{10}} d^{\frac{5}{2}}$ and the rest have size between $\ell$ and $\ell \log d = \frac{1}{c_{10}} d^\frac{3}{2}\log d$. Let us call this collection of pieces $\cc{C}$, noting that $V\big(\bigcup \cc{C}\big) = Y$. 

We next show that whp for any partition $\cc{C}=\cc{C}_A \cup \cc{C}_B$ of $\cc{C}$ into two parts, each covering a positive proportion of the vertices of $Y$, we expose exponentially many edge-disjoint paths between $A$ and $B$ in the second sprinkling step. This argument will require that the pieces in $\cc{C}$ are not significantly smaller than $d^{\frac{3}{2}}$, whereas later bounds will be optimised by taking the pieces in $\cc{C}$ as small as possible, which motivates our choice of $\ell$ in the above.

\begin{claim}\label{c:manypaths}
Whp for any partition $\cc{C} = \cc{C}_A \cup \cc{C}_B$ of $\cc{C}$ into two parts, with $A = V\big(\bigcup \cc{C}_A\big)$ and $B = V\big(\bigcup \cc{C}_B \big)$, such that $\min\{|A|,|B|\} \geq c_7 n$, there are at least $c_{11} nd^{-\frac{3}{2}}$ edge-disjoint paths between $A$ and $B$ in  $Q^d_{q_2}$.
\end{claim}
\begin{proof}[Proof of Claim \ref{c:manypaths}]
The total number of pieces in $\cc{C}$ is at most $c_{10}nd^{-\frac{3}{2}}$ and so the total number of possible partitions is at most $2^{c_{10}nd^{-\frac{3}{2}}}$. Hence, if we can show that the claim does not hold for a fixed partition with probability at most $\exp \left(- c_{10}nd^{-\frac{3}{2}}\right)$, then whp the claim holds for all partitions by the union bound.

Given a fixed partition $\cc{C}_A \cup \cc{C}_B$ with $A$ and $B$ as above, let $N(A) = A \cup \partial_{v,q_1} (A)$ be the inclusive neighbourhood of $A$ in $Q_1$ and similarly let $N(B)$ be the inclusive neighbourhood of $B$ in $Q_1$. Let $D = N(A) \cap N(B)$. We split into two cases:
\begin{enumerate}[(i)]
\item\label{i:Dlarge} $|D| \geq c_8 nd^{-\frac{1}{2}}$;
\item\label{i:Dsmall} $|D| < c_8 nd^{-\frac{1}{2}}$.
\end{enumerate}

{\bf Case \eqref{i:Dlarge}:} Let us suppose first that $|D| \geq c_8 nd^{-\frac{1}{2}}$. In this case, let $D' = D \setminus X$, where $X$ is as in \eqref{e:X}, so that $|D'| \geq \frac{|D|}{2}$, since $|X| \leq n^{1-c_5}$. 

Since each $x \in D'$ is in $D$, $x$ is either in $A$ or has a neighbour in $A$ and similarly $x$ is either in $B$ or has a neighbour in $B$. Moreover, since $x \not\in X$ and $A \cup B = Y$, it follows that $x$ has at least $\frac{c_4 d}{2}$ many neighbours in one of $A$ or $B$. For each $x \in D'$, let us fix a set $E_x$ of edges witnessing this, consisting of $\frac{c_4 d}{2}$ edges from $x$ to one of the vertex sets $A$ or $B$ together with a single edge from $x$ to the other vertex set if $x$ is not already an element of that set (see Figure \ref{f:case1}).

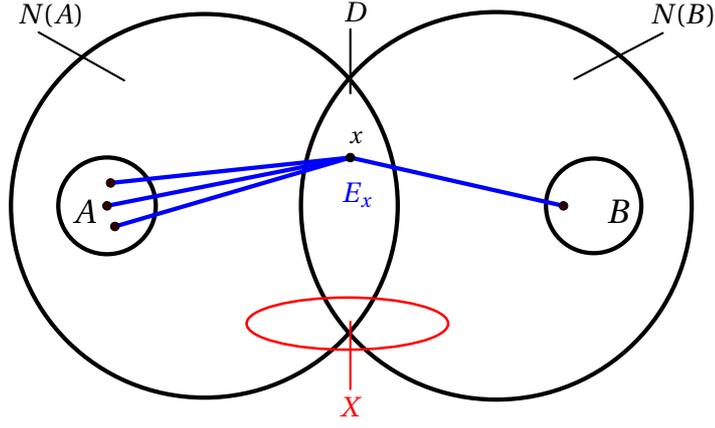
\begin{figure}[ht!]
\center
\scalebox{0.8}{
\begin{tikzpicture}[line cap=round,line join=round,>=triangle 45,x=1cm,y=1cm, scale=0.8]
\draw [line width=2pt] (-3,0) circle (3.973726211001457cm);
\draw [line width=2pt] (3,0) circle (4.014635724446242cm);
\draw [line width=2pt] (-5,0) circle (1.0018482919085105cm);
\draw [line width=2pt] (5,0) circle (0.9800510190801291cm);
\draw (-.2,1.7) node[anchor=north west] {$x$};
\draw (-0.3,4.4) node[anchor=north west] {\large $D$};
\draw [line width=1pt] (0,2.33)-- (0,3.66);
\draw (-5.9,0.3) node[anchor=north west] {\Large $A$};
\draw (5.1,0.3) node[anchor=north west] {\Large $B$};
\draw [rotate around={0:(-0.06,-2.44)},line width=1.2pt,color=ffqqqq] (-0.06,-2.44) ellipse (2.0716606002141353cm and 0.5401644587341691cm);
\draw [color=ffqqqq](-0.4,-3.8) node[anchor=north west] {\large $X$};
\draw [line width=1pt,color=ffqqqq] (0,-2.4)-- (0,-3.8);
\draw [line width=2pt,color=qqqqff] (0,1)-- (4.38,0);
\draw [line width=2pt,color=qqqqff] (0,1)-- (-4.926,0.474);
\draw [line width=2pt,color=qqqqff] (0,1)-- (-5,0);
\draw [line width=2pt,color=qqqqff] (0,1)-- (-4.838,-0.428);
\draw [color=qqqqff](-0.34273971341581727,0.6420739200889095) node[anchor=north west] {\large $E_x$};
\draw (-7,4.35) node[anchor=north west] {\large $N(A)$};
\draw (6,4.35) node[anchor=north west] {\large $N(B)$};
\draw [line width=1pt] (-6.4062538711037,3.5956574004507864)-- (-4.647077751919517,2.598790932913084);
\draw [line width=1pt] (6.421071997947632,3.580997599457585)-- (4.603256674790643,2.628110534899487);
\begin{scriptsize}
\draw [fill=ttqqqq] (-5,0) circle (2.5pt);
\draw [fill=black] (0,1) circle (2.5pt);
\draw [fill=ttqqqq] (4.38,0) circle (2.5pt);
\draw [fill=ttqqqq] (-4.926,0.474) circle (2.5pt);
\draw [fill=ttqqqq] (-4.838,-0.428) circle (2.5pt);
\end{scriptsize}
\end{tikzpicture}
}
\caption{A vertex $x \in D'=D\setminus X$ with corresponding edge set $E_x$ consisting of $\frac{c_4 d}{2}$ edges from $x$ to $A$ and an edge from $x$ to $B$}\label{f:case1}
\end{figure}

Let us choose a subset $D'' \subseteq D'$ of size at least $\frac{|D'|}{2}$ such that all $x \in D''$ have the same parity, and hence the edge sets $E_x$ are disjoint for each $x \in D''$.

Then, for each $x \in D''$, there is a path between $A$ and $B$ using edges from $E_x$ in $Q^d_{q_2}$ with probability at least 
\[
q_2 \left(1 - (1-q_2)^{\frac{c_4 d}{2}}\right)\geq q_2 \left( 1- \exp\left( -\frac{q_2 c_4 d}{2}\right)\right) \geq \frac{c_6\left(1- \exp\left( - \frac{c_4 c_6}{2} \right)\right)}{d} \geq \frac{2c_9}{d},
\]
since $1-x \leq e^{-x}$ for all $x>0$ and $q_2 \geq \frac{c_6}{d}$.

Since the sets $E_x$ are disjoint, these events are independent for different $x \in D''$ and so, by Chernoff type bounds (Lemma \ref{l:Chernoff} \eqref{i:chernoff1}), we have that there are less than $\frac{c_9}{d} |D''| \geq  c_{11} nd^{-\frac{3}{2}}$ such paths with probability at most $2\exp \left( - \frac{c_8c_9 nd^{-\frac{3}{2}}}{32}\right) \leq \exp \left(- c_{10}nd^{-\frac{3}{2}}\right)$.

{\bf Case \eqref{i:Dsmall}:} Let us now deal with the second case, where $|D| < c_8 nd^{-\frac{1}{2}}$. In this case we note that $|N(A)| \geq |A| \geq c_7 n$ and $|N(A) \cap B| \leq |D| \leq c_8  n d^{-\frac{1}{2}}$, and so $|N(A)^c| \geq |B| - |D| \geq \frac{c_7}{2} n$. Therefore, it follows from Theorem \ref{t:vtxiso} that there is a matching $F$ in $Q^d$ of size at least $3 c_8 nd^{-\frac{1}{2}}$ from $N(A)$ to $N(A)^c$.

Since $|X| \ll |D| < c_8 nd^{-\frac{1}{2}}$, if we let $F'$ be those edges in $F$ which do not have an endpoint in either $D$ or $X$ then $|F'| \geq c_8 nd^{-\frac{1}{2}}$. For each edge $uv=e$ in $F'$, where without loss of generality $u \in N(A) \setminus (D \cup X)$, either $u \in A$ or $u$ has at least $c_4 d$ many neighbours in $A$ and similarly either $v \in B$ or $v$ has at least $c_4 d$ many neighbours in $B$. For each $e \in F'$, let us fix a set of edges $E_e$ witnessing this, consisting of $e$ together with  $c_4 d$ edges from $u$ to $A$ if $u \not\in A$ and $c_4d$ edges from $v$ to $B$ if $v \not\in B$ (see Figure \ref{f:case2}).

\begin{figure}[ht!]
\center
\scalebox{0.85}{
\begin{tikzpicture}[line cap=round,line join=round,>=triangle 45,x=1cm,y=1cm, scale=0.8]
\clip(-7.316184854367966,-5.036459304646435) rectangle (9.465115794760276,5.897698261161853);
\draw [line width=2pt] (3,0) circle (3.3842874192853127cm);
\draw [line width=2pt] (-3,0) circle (3.381421569359371cm);
\draw (1.25,.2) node[anchor=north west] {\large $D$};
\draw [line width=1pt] (1.3,-.1)-- (0,.3);
\draw [rotate around={0:(0,-1.5)},line width=1.2pt,color=ffqqqq] (0,-1.5) ellipse (1.2cm and 0.4cm);
\draw [color=ffqqqq](1.45,-1.8) node[anchor=north west] {\large $X$};
\draw [line width=1pt,color=ffqqqq] (1.5,-2.05)-- (0.6,-1.44);
\draw [line width=2pt,color=qqqqff] (-1.3816226920812198,2.4970837563675694)-- (1.4117266172991565,2.4939408427056775);
\draw (-1.4,2.3) node[anchor=north west] {$u$};
\draw (0.9,2.3) node[anchor=north west] {$v$};
\draw [line width=2pt] (-5,0) circle (1.014059040974722cm);
\draw [line width=2pt] (5,0) circle (0.9973217286593371cm);
\draw (-5.9,0.4) node[anchor=north west] {\Large $A$};
\draw (5.2,0.4) node[anchor=north west] {\Large $B$};
\draw [line width=2pt,color=qqqqff] (-1.3816226920812198,2.4970837563675694)-- (-5.003044898313415,0.6256235048101039);
\draw [line width=2pt,color=qqqqff] (-5,0)-- (-1.3816226920812198,2.4970837563675694);
\draw [line width=2pt,color=qqqqff] (-1.3816226920812198,2.4970837563675694)-- (-4.989038340517655,-0.6069535812168142);
\draw [line width=2pt,color=qqqqff] (1.4117266172991565,2.4939408427056775)-- (4.997637367859544,0.5836038314228226);
\draw [line width=2pt,color=qqqqff] (1.4117266172991565,2.4939408427056775)-- (5,0);
\draw [line width=2pt,color=qqqqff] (1.4117266172991565,2.4939408427056775)-- (4.983630810063783,-0.5789404656252933);
\draw [color=qqqqff](-0.3,3.4) node[anchor=north west] {\large $E_e$};
\draw (-6.2,3.75) node[anchor=north west] {\large $N(A)$};
\draw (5,3.75) node[anchor=north west] {\large $N(B)$};
\draw [line width=1pt] (-5.5,3)-- (-4.5,2.5);
\draw [line width=1pt] (5.5,3)-- (4.5,2.5);
\begin{scriptsize}
\draw [fill=ttqqqq] (-1.3816226920812198,2.4970837563675694) circle (2.5pt);
\draw [fill=ttqqqq] (1.4117266172991565,2.4939408427056775) circle (2.5pt);
\draw [fill=ttqqqq] (-5.003044898313415,0.6256235048101039) circle (2.5pt);
\draw [fill=ttqqqq] (-5,0) circle (2.5pt);
\draw [fill=ttqqqq] (-4.989038340517655,-0.6069535812168142) circle (2.5pt);
\draw [fill=ttqqqq] (4.997637367859544,0.5836038314228226) circle (2.5pt);
\draw [fill=ttqqqq] (5,0) circle (2.5pt);
\draw [fill=ttqqqq] (4.983630810063783,-0.5789404656252933) circle (2.5pt);
\end{scriptsize}
\end{tikzpicture}
}
\caption{An edge $uv=e \in F'$ with corresponding edge set $E_e$ consisting of $c_4d$ edges from $u$ to $A$ and $c_4d$ edges from $v$ to $B$.}\label{f:case2}
\end{figure}
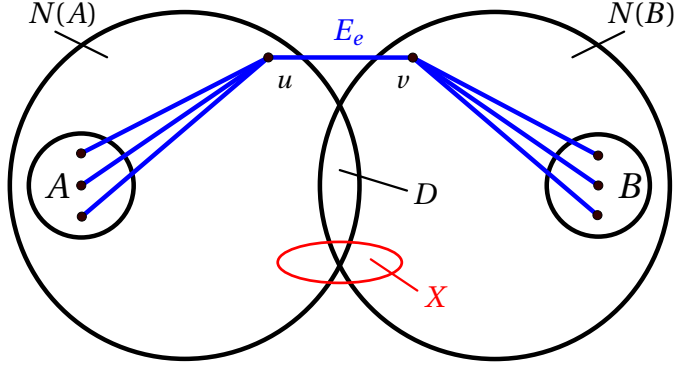

Then, for each $uv=e \in F'$, there is a path in $Q^d_{q_2}$  between $A$ and $B$ using edges from $E_e$ with probability at least
\[
q_2 \left(1-(1-q_2)^{c_4 d}\right)^2 \geq q_2 \left(1-\exp\left(-q_2c_4 d \right)\right)^2 \geq \frac{c_6 \left(1-\exp\left(-c_4 c_6\right)\right)^2}{d} \geq \frac{2 c_9}{d}.
\]
Since the sets $E_e$ are disjoint, these events are independent for different $e\in F'$ and so, by Chernoff type bounds (Lemma \ref{l:Chernoff} \eqref{i:chernoff1}), we have that there are less than $\frac{c_9}{d} |F'| \geq c_{11} nd^{-\frac{3}{2}}$ such paths with probability at most $2\exp \left( - \frac{c_8 c_9 nd^{-\frac{3}{2}}}{16} \right) \leq \exp \left(- c_{10}n d^{-\frac{3}{2}}\right)$.
\end{proof}

Assuming Claims \ref{c:giant} and \ref{c:manypaths}, our aim is to show that the edge boundary of every subset $S \subseteq L_1$ satisfying ${\alpha n \leq |S| \leq \frac{|V(L_1)|}{2}}$ is large. Let us fix some such subset $S$.

\begin{claim}\label{c:largeedgeboundary}
Whp $\left| \partial_{e,p}(S) \right| \geq \frac{c_{12} n}{d^3 (\log d)^2}$.
\end{claim}
\begin{proof}[Proof of Claim \ref{c:largeedgeboundary}]
We first note that by Claim \ref{c:giant} whp $|S \cap  Y| \geq (\alpha- 4c_1)n$. Let us write
\begin{align*}
\cc{C}_0 &= \{ C \in \cc{C} \colon S \cap C \neq \emptyset \text{ and } C \setminus S \neq \emptyset \},\\
\cc{C}_A &= \{C \in \cc{C} \colon C \subset S\},\\
\cc{C}_B &= \{C \in \cc{C} \colon S \cap C = \emptyset\},
\end{align*}
and let $A = V\big(\bigcup \cc{C}_A\big)$ and $B =V\big(\bigcup \cc{C}_B \cup \bigcup \cc{C}_0 \big)$. 

Since each $C \in \cc{C}$ is connected in $Q_1\subseteq Q_2$, there are at least $|\cc{C}_0|$ edges in the edge boundary of $S$ in $Q_2$. We split into two cases:
\begin{enumerate}[(i)]
\item\label{i:Clarge} $|\cc{C}_0| \geq \frac{c_{12} n}{ d^3 (\log d)^2}$;
\item\label{i:Csmall} $|\cc{C}_0| < \frac{c_{12} n}{ d^3 (\log d)^2}$.
\end{enumerate}

{\bf Case \eqref{i:Clarge}:} Suppose first that $|\cc{C}_0| \geq \frac{c_{12} n}{ d^3 (\log d)^2}$. Since each $C \in \cc{C}_0$ is connected in $Q_1$ and meets both $S$ and $V\left(Q^d\right) \setminus S$, it follows that there are at least $\frac{c_{12} n}{d^3 (\log d)^2}$ many edges in the edge boundary of $S$ in $Q_1$ and hence, at least this many in the edge boundary of $S$ in $Q_2$. 

{\bf Case \eqref{i:Csmall}:} Suppose instead that $|\cc{C}_0| < \frac{c_{12} n}{ d^3 (\log d)^2}$. Since there are at most $nd^{-5}$ pieces in $\cc{C}$ of size between $\frac{1}{c_{10}} d^{\frac{3}{2}}$ and $\frac{1}{c_{10}} d^{\frac{5}{2}}$ and the rest have size at most $\frac{1}{c_{10}} d^{\frac{3}{2}} \log d$, it follows that, if we let $Z:=V\big(\bigcup \cc{C}_0\big)$, then
\begin{equation}\label{e:Zbound}
|Z| \leq \frac{2 c_{12} n}{c_{10} d^{\frac{3}{2}} \log d}.
\end{equation}

In particular, we have that $|A| \geq |S| - |Z| \geq c_7 n$ and $|B| \geq |Y| - |S| \geq c_7 n$, and so, by Claim \ref{c:manypaths}, whp there is a family $\cc{P}$ of at least $c_{11} nd^{-\frac{3}{2}}$ many edge disjoint paths between $A$ and $B$ in $Q^d_{q_2}$. 

Then, by Lemma \ref{l:edgeboundary} applied with $k = \frac{5}{2}$, whp either $|Z| \leq \frac{c_{11}}{2}nd^{-\frac{5}{2}}$, or $Z$ has edge boundary satisfying
\[
\left|\partial_{e,p}(Z)\right| \leq |Z| \log d \leq \frac{2c_{12}}{c_{10}} n d^{-\frac{3}{2}} \leq \frac{c_{11}}{2}nd^{-\frac{3}{2}},
\]
where we used \eqref{e:Zbound} to bound $|Z|$ from above.
In either case, since every vertex in $Z$ has degree at most $d$, at most half of the paths in $\cc{P}$ meet $Z$, and every other path in $\cc{P}$ must contribute at least one edge to the boundary of $S$ in $Q_2$, since $A \subseteq S$ and $(B \setminus Z) \cap S = \emptyset$. It follows that the edge boundary of $S$ in $Q_2$ has size at least $\frac{c_{11}}{2}nd^{-\frac{3}{2}} \geq \frac{c_{12} n}{d^3 (\log d)^2}$.
\end{proof}

It follows from Claim \ref{c:largeedgeboundary} that the theorem holds with $\beta = c_{12}$.
\end{proof}

\begin{proof}[Proof of Corollary \ref{c:vertexexpansion}]
Let $\beta$ be given by Theorem \ref{t:edgeexpansion}. By Lemma \ref{l:edgeboundary}, we have that whp every subset $X \subseteq V\left(Q^d\right)$ of size $|X| \geq n d^{-5}$ has edge boundary in $Q^d_p$ such that
\[
\left|\partial_{e,p}(X)\right| \leq |X| \log d.
\]

So, we may assume that the above and the conclusion of Theorem \ref{t:edgeexpansion} hold. Let $S \subseteq L$ be such that ${\alpha n \leq |S| \leq \frac{|V(L_1)|}{2}}$ and let $X = \partial_{v,p}(S)$ be the vertex boundary of $S$ in $Q^d_p$. Then, by our choice of $\beta$, we have $\left|\partial_{e,p}(S)\right| \geq \beta  \frac{n}{d^3 (\log d)^2}$ and so, since $Q^d$ is $d$-regular, it follows that $|X| \geq \frac{n}{d^4 (\log d)^2} \geq n d^{-5}$. Hence, $\left|\partial_{e,p}(X)\right| \leq  |X| \log d$.

However, since each edge in the edge boundary of $S$ is in the edge boundary of $X$ we obtain
\[
  |X| \log d \geq \left|\partial_{e,p}(X)\right| \geq \left|\partial_{e,p}(S)\right| \geq \beta \frac{n}{d^3 (\log d)^2}.
\]
It follows that $|X| \geq \beta \frac{n}{d^3 (\log d)^3}$, and so the corollary follows
\end{proof}

As corollaries, we can deduce Theorems \ref{t:longcycle} and \ref{t:Hadwiger}.

\begin{proof}[Proof of Theorem \ref{t:longcycle}]
By Corollary \ref{c:vertexexpansion} given $\alpha \ll \epsilon$ there exists a $\beta >0$ such that whp every subset $S$ of the largest component $L_1$ of $Q^d_p$ with $\alpha n \leq |S| \leq \frac{|V(L_1)|}{2}$ satisfies $\left|\partial_{v,p}(S)\right| \geq  \frac{\beta n}{d^3 (\log d)^3}$. Hence, applying Theorem \ref{t:cycleexpander} with $k = \frac{|V(L_1)|}{2}$ and $t = \frac{\beta n}{d^3 (\log d)^3}$ we can conclude that $L_1$ contains a cycle of length $\Omega\left( \frac{ n}{d^3 (\log d)^3} \right)$.
\end{proof}

\begin{proof}[Proof of Theorem \ref{t:Hadwiger}]
If the largest component $L_1$ of $Q^d_p$ does not contain a $K_t$-minor, then by Theorem \ref{t:minorexpander} there is some constant $C >0$ such that $L_1$ contains a subset $S$ of size at most $C t \sqrt{|V(L_1)|} \leq Ct \sqrt{n}$, such that each component of $G \setminus S$ has order at most $\frac{2|V(L_1)|}{3}$. It follows that we can find some subset of $L_1$ of size between $\frac{|V(L_1)|}{3}$ and $\frac{|V(L_1)|}{2}$ whose vertex boundary is contained in $S$, and hence has size at most $Ct\sqrt{n}$. 

However, by Corollary \ref{c:vertexexpansion}, whp every subset of $L_1$ of size between $\frac{1}{3}|V(L_1)|$ and $\frac{1}{2}|V(L_1)|$ has a vertex boundary of size at least $\frac{\beta n}{d^3 (\log d)^3}$. It follows that $t \geq\frac{\beta \sqrt{n}}{ Cd^3 (\log d)^3 } = \Omega \left( \frac{\sqrt{n}}{d^3 (\log d)^3}\right)$.
\end{proof}

\section{Genus of the largest component}\label{s:genus}
Given a graph $G$ let us write $g(G)$ for the \emph{genus} of $G$, the smallest $k \in \mathbb{N}$ such that $G$ can be embedded on an orientable surface of genus $k$.

Since the genus of $K_t$ is known to be $\Omega\left(t^2\right)$, and the genus of a graph is a \emph{minor monotone function}, that is, if $H$ is a minor of $G$ then $g(H) \leq g(G)$, it follows from Theorem \ref{t:Hadwiger} that whp $g\left(Q^d_p\right) = \Omega\left(\frac{n}{d^3 (\log d)^3} \right)$ in the supercritical regime. However, if we are more careful, we can improve this to a bound of optimal asymptotic order.

\begin{proof}[Proof of Theorem \ref{t:genus}]
Euler's formula tells us that the genus of a connected graph $L$ satisfies
\begin{equation}\label{e:euler}
g(L) = \frac{1}{2}\left(e(L) - v(L) - f(L) +2\right) = \frac{1}{2}\left(\text{excess}(L) - f(L) +2\right),
\end{equation}
where $e(L)$ and $v(L)$ are the number of edges and vertices of $L$ respectively, excess$(L):= e(L) - v(L)$ is the \emph{excess} of $L$ and $f(L)$ is the number of faces of $L$ when embedded on a surface of minimal genus. We will show that whp the largest component $L_1$ of $Q^d_p$ satisfies
\begin{enumerate}[(i)]
\item\label{i:excess}excess$(L_1) = \Omega(n)$; and
\item\label{i:faces}$f(L_1) = o(n)$,
\end{enumerate}
from which it follows, from \eqref{e:euler}, that whp $g(L_1)$ is also $\Omega(n)$.

{\bf Property \eqref{i:excess}:} To prove that whp excess$(L_1) = \Omega(n)$, we again argue via a two-round exposure. Let $q_1 = \frac{1+\frac{\epsilon}{2}}{d}$, so that $q_2 = \frac{p-p_1}{1-p_1} \geq \frac{\epsilon}{2d}$ and let us write as before $Q_1 = Q^d_{q_1}$ and $Q_2 = Q_1 \cup Q^d_{q_2}$ where the two random subgraphs are independent, so that $Q_2$ has the same distribution as $Q^d_p$.

By Theorem \ref{t:AKS} there exists $\alpha>0$ such that whp there is a component $L'_1$ of $Q_1$ with ${|V(L_1')| \geq \alpha n}$. Since $L'_1$ is connected it contains at least $|V(L_1')| -1$ edges of $Q_1$.

However, if we let $\beta \ll \alpha$ and let $\cc{A}$ be the event that there exists a set of vertices ${W \subseteq V\left(Q^d\right)}$ of size at least $\alpha n$ such that $W$ spans at most $\beta d |W|$ edges in $Q^d$ and $W$ spans at least $|W|-1 \geq \frac{|W|}{2}$ edges of $Q_1$, then it follows, by Lemma \ref{l:Chernoff} \eqref{i:chernoff2}, that
\begin{align*}
\mathbb{P}(\cc{A}) &\leq \sum_{i \geq \alpha n} \binom{n}{i}\mathbb{P}\left(\text{Bin}\left(\left\lfloor \beta d i \right\rfloor, \frac{2}{d}\right) \geq \frac{i}{2}\right) \leq \sum_{i \geq \alpha n} \left(\frac{en}{i}\right)^i (2\beta e)^{\frac{i}{2}} = \sum_{i \geq \alpha n} \left(\frac{\sqrt{2} e^{\frac{3}{2}} \sqrt{\beta}}{\alpha} \right)^i =o(1).
\end{align*}

Hence, we can conclude that whp $L'_1$ spans at least $\beta d |V(L'_1)| \geq \alpha \beta d n$ many edges of $Q^d$.  Let $T$ be a spanning tree of $L'_1$ in $Q_1$, then $L'_1$ spans at least $\alpha \beta d n - |V(T)| \geq \frac{\alpha \beta d n}{2}$ edges in $Q^d - T$.

Hence, again by Chernoff type bounds (Lemma \ref{l:Chernoff} \eqref{i:chernoff1}), we have that the number of edges spanned by $L'_1$ in $Q^d_{q_2} \setminus T$, which we denote by $Y$, is such that
\[
\mathbb{P}\left(Y \leq \frac{\alpha \beta \epsilon n}{8}\right) \leq \mathbb{P}\left( \text{Bin}\left( \left\lceil\frac{\alpha \beta d n}{2}\right\rceil, \frac{\epsilon}{2d} \right) \leq \frac{\alpha \beta \epsilon n}{8}\right) \leq 2\exp\left(-\frac{\alpha \beta \epsilon n}{32}\right) = o(1).
\]

So, we can conclude that whp there is some component $L_1 \supseteq L'_1$ in $Q_2 = Q_1 \cup Q^d_{q_2}$ which has excess at least $\frac{\alpha \beta \epsilon n}{8}$. Note that, by Theorem \ref{t:AKSfine} whp there is a unique linear sized component in $Q_2$ and so whp $L_1$ is in fact the largest component in $Q_2$. Summing up, whp
\begin{equation}\label{e:excess}
\text{excess}(L_1) \geq \frac{\alpha \beta \epsilon n}{8}.
\end{equation}
We note that the typical existence of a component of large excess in $Q^d_p$ can also be deduced in a relatively straightforward manner from Theorem \ref{t:AKSfine}. Indeed, by Theorem \ref{t:AKSfine} almost every vertex in $Q^d$, and so almost every vertex in the giant component $L$ of $Q_1$, has linearly many (in $d$) neighbours in $L$. It will then follow that whp after sprinkling we add linearly many (in $n$) edges to the vertex set of $L$ in $Q_2$.

{\bf Property \eqref{i:faces}:} To prove that whp $f(L_1) = o(n)$, let us consider the expected number of cycles of length $s$ in $Q^d_p$ for a fixed $s$. We note that we can naively bound the number of cycles of length $s$ in $Q^d$ by $2nd^{s-2}$. Indeed, for any vertex $x$ there are at most $d^{s-2}$ many paths of length $s-2$ starting at $x$. However, since the codegree of any pair of vertices in $Q^d$ is at most two, there are at most two ways to complete any path of length $s-2$ to a cycle of length $s$. Since there are at most $n$ many ways to choose the initial vertex $x$, the bound follows. We mention that this bound is far from tight and indeed, without too much effort, it can be shown that the number of cycles of length $s$ in $Q^d$ is at most $2^d\left(\frac{sd}{2}\right)^{\frac{s}{2}}$ (see \cite[Claim 3.2]{MS18}).

Since each such cycle is in $Q^d_p$ with probability $p^s$, and $p = \frac{1+\epsilon}{d}$, it follows that the number of cycles of length $s$, which we denote by $X_s$, satisfies 
\[
\bb{E}(X_s) \leq 2nd^{s-2} p^s = 2n d^{-2}(1+\epsilon)^s.
\]
Hence, the expected number of cycles of length at most $\sqrt{\log d}$ in $Q^d_p$ is at most
\[
\sum_{s=1}^{\sqrt{\log d}} \bb{E}(X_s) \leq 2nd^{-2}\sum_{s=1}^{\sqrt{\log d}} (1+\epsilon)^s \leq 2nd^{-2}\sqrt{\log d} 2^{\sqrt{\log d}}\leq nd^{-1}.
\]
Hence, by Markov's inequality, whp the number of cycles of length at most $\sqrt{\log d}$ in $Q^d_p$ is $o(n)$. It follows that whp the number of faces of length at most $\sqrt{\log d}$ in any embedding of $L_1$ is $o(n)$. Furthermore, since in any embedding each edge is in at most two faces, the number of faces of length at least $\sqrt{\log d}$ in any embedding of $L_1$, which we denote by $Y_{\geq \sqrt{\log d}}$, is at most $\frac{2e(L_1)}{\sqrt{\log d}} \leq \frac{2e\left(Q^d_p\right)}{\sqrt{\log d}}$. However, again by Lemma \ref{l:Chernoff}, we have that whp $e\left(Q^d_p\right) \leq dnp = (1+\epsilon)n$ and so whp
\[
Y_{\geq \sqrt{\log d}} \leq \frac{2e\left(Q^d_p\right)}{\sqrt{\log d}} \leq \frac{2(1+\epsilon)n}{\sqrt{\log d}} = o(n).
\]
Hence, whp $f(L_1) = o(n)$ and so, by Euler's formula \eqref{e:euler} applied to $L_1$, we have
\[
g\left(Q^d_p\right) \geq g(L_1) \geq \frac{1}{2}\left(\text{excess}(L_1) - f(L_1) +2\right) \stackrel{\eqref{e:excess}}{\geq} \frac{1}{2}\left(\frac{\alpha \beta \epsilon n}{8} - o(n) \right) = \Omega(n).
\]
\end{proof}

\section{Discussion}\label{s:discussion}
In Theorems \ref{t:longcycle} and \ref{t:Hadwiger} we have shown that in the supercritical regime when $p=\frac{1+\epsilon}{d}$ for a positive constant $\epsilon$, whp $Q^d_p$ contains a cycle whose length is \emph{almost linear} in $n=2^d$, that is, up to some polylogarithmic term in $n$, and similarly a complete minor of order almost $\sqrt{n}$. However, it seems unlikely that these results are best possible, and analogous to the case of $G(d+1,p)$, it is natural to conjecture that for supercritical $p$ whp $Q^d_p$ contains a cycle whose length is in fact \emph{linear} in $n$, and a complete minor of order $\sqrt{n}$.
\begin{question}Let $\epsilon>0$ and $p=\frac{1+\epsilon}{d}$.
\begin{enumerate}
\item Is it the case that whp $Q^d_p$ contains a cycle of length $\Omega(n)$?
\item Is it the case that whp $Q^d_p$ contains a complete minor of order $\Omega\left(\sqrt{n}\right)$?
\end{enumerate}
\end{question}
Furthermore, in both questions it would also be interesting to know the dependence of the results on $\epsilon$. For example, in $G(d+1,p)$ it is known that for $p = \frac{1+\epsilon}{d}$ the largest component is of order $(2\epsilon + o(\epsilon))d$, the length of the longest cycle is of order $\Theta\left(\epsilon^2\right) d$ (see, for example, \cite[Theorem 5.7]{JLR00}) and the order of the largest complete minor is $\Theta\left(\epsilon^\frac{3}{2}\right) \sqrt{d}$ (see \cite{FKO09}).

As indicated in the introduction, it seems unlikely that such results can be proven simply by considering the expansion properties of the largest component. However, it is still an interesting question precisely how strong expansion properties of the largest component we can guarantee whp. For example, Krivelevich \cite{K18} showed that in the supercritical regime $p= \frac{1+\epsilon}{d}$ in $G(d+1,p)$, whp the largest component contains a subgraph of order $\Theta(d)$ which is a bounded degree $\alpha$-expander for some $\alpha =\alpha(\epsilon) > 0$, where a graph $G$ is an \emph{$\alpha$-expander} if for every subset $W \subseteq V(G)$ with $|W| \leq \frac{|V(G)|}{2}$ we have that $\left|\partial_v(W)\right| \geq \alpha |W|$.

\begin{question}\leavevmode
Let $\epsilon >0$ and $p = \frac{1+\epsilon}{d}$. For what $\alpha =\alpha(d)$ is it true that whp the largest component of $Q^d_p$ contains a subgraph of order $\Theta(n)$ which is an $\alpha$-expander?
\end{question}
We note that by a similar argument as in \cite[Lemma 2.4]{K19}, Corollary \ref{c:vertexexpansion} implies that we can take $\alpha(d) = \Omega\left(\frac{1}{d^3 (\log d)^3}\right)$, however it seems unlikely that this is optimal.

A careful examination of the proof of Theorem \ref{t:genus} shows that whp the genus of the largest component $L_1$ of $Q^d_p$ is asymptotically equal to half of the excess of $L_1$. In the case of ${G(d+1,p)}$, in the supercritical regime, it can be shown that if we delete the largest component, what is left has approximately the distribution of a subcritical random graph, and so whp all other components are unicyclic or trees. From this it follows that the genus of $G(d+1,p)$ is asymptotically equal to the genus of its largest component. Furthermore, by Euler's formula for an arbitrary graph $G$
\[
g(G) = \frac{1}{2}\left(e(G) - v(G) - f(G) + \kappa(G) + 1\right),
\]
where $\kappa(G)$ is the number of components of $G$, and by counting carefully the number of tree-components it is possible to determine asymptotically the genus of $G(d+1,p)$ in this regime, and also the excess of the giant component (see \cite[Theorems 1.2 and 1.3]{DKK20}).

However, it is not the case that whp $Q^d_p$ is planar for an arbitrary subcritical $p$. For example, it is easy to see that whp $Q^4 \subseteq Q^d_p$ when $p= \Theta\left(\frac{1}{d}\right)$, and $Q^4$ is non-planar. So, even if such a `symmetry rule' were to hold in this model, it would not be immediate that the main contribution to the genus of $Q^d_p$ is coming from the excess of the largest component. It would be interesting to determine asymptotically the genus of $Q^d_p$ in the supercritical regime.

There are some interesting open questions about the model $Q^d_p$ in the paper of Condon, Espuny D{\'\i}az, Girao, K{\"u}hn and Osthus \cite{CDGKO21}. In particular, they used as a crucial part of their proof the fact that whp $Q^d_{\frac{1}{2}}$ contains an `almost spanning' path, that is, a path containing $(1-o(1))n$ vertices, and they showed that this property is in fact true for $Q^d_p$ for any constant $p$. 

However, analogous to the case of $G(d+1,p)$, we should perhaps expect such a path to exist for much smaller values of $p$. In particular, if we expect the sparse random subgraph $Q^d_{p}$ with $p=\frac{c}{d}$ to contain a path of linear length $f(c)n$ for some function $f(c)$ when $c > 1$, it is natural to conjecture that $f(c) \to 1$ as $c \rightarrow \infty$.

\begin{question}
Let $p = \omega\left(\frac{1}{d}\right)$. Is it true that whp $Q^d_p$ contains a path of length $(1-o(1))n$?
\end{question}

Finally, other notions of random subgraphs of the hypercube  $Q^d$ have also been studied. In particular, if we let $Q^d(p)$ denote a random \emph{induced} subgraph of $Q^d$ obtained by retaining each \emph{vertex} independently with probability $p$, then the typical existence of a giant component in $Q^d(p)$ when $p=\frac{1+\epsilon}{d}$ for a fixed $\epsilon >0$ was shown by Bollob\'{a}s, Kohayakawa and {\L}uczak \cite{BKL94}, and this was extended to a broader range of $p$ with $\epsilon=o(1)$ by Reidys \cite{R09}. It would be interesting to know if the giant component in $Q^d(p)$ whp also has good expansion properties.

	\bibliographystyle{plain}
	\bibliography{cube}
	
\appendix
\section{Proof of Theorem \ref{t:AKSfine}}\label{a:AKS}
We note that the proof below follows the proof in \cite{AKS81} quite closely. There is a slight error in the original proof, where the vertex-isoperimetric inequality of Harper \cite{H66} is used to assert the existence of a large matching between a subset of the hypercube and its complement, however, the claimed size of this matching is too large by a factor of $d$. This error is easily fixed, and we do so by utilising instead the stronger isoperimetric result given in Theorem \ref{t:vtxiso}.

\begin{proof}[Proof of Theorem \ref{t:AKSfine}]
During the proof we will introduce a series of constants $c_1,c_2,\ldots,$ with the understanding that each $c_i$ will be as small as is necessary for the argument in terms of $\delta_1,\delta_2$ and all $c_j$ with $j<i$. Recall that $p=\frac{1+\delta_1}{d}$ and $\gamma_1$ denotes the survival probability of the Po$(1+ \delta_1)$ branching process.

Let us first show that whp $|V(L_1)| \leq (\gamma_1 + \delta_2)n$. To begin with, we note that the number of vertices in the component of $Q^d_p$ containing a fixed vertex $v$ of the hypercube is dominated by a Bin$(d,p)$ branching process, which tends in distribution to a Po$(1+\delta_1)$ branching process as $d \to \infty$.

\begin{claim}\label{c:upperboundgiant}
Whp the number of vertices contained in component of order at most $d$ in $Q^d_p$ is at least $(1-\gamma_1 - c_1)n$.
\end{claim}
\begin{proof}[Proof of Claim \ref{c:upperboundgiant}]
It follows from standard bounds, see for example \cite{BR12}, that the probability that a Po$(1+\delta_1)$ branching process dies after growing to size $d$ is $o(1)$, and so the probability that $v$ is contained in a component of order at most $d$ is at least $1 - \gamma_1 - c_2$. If we let $Z$ be the number of vertices contained in components of order at most $d$ in $Q^d_p$, then it follows that $\mathbb{E}(Z) \geq (1-\gamma_1 - c_2)n$. However, $Z$ is a $2d$-Lipschitz function of the edges of $Q^d_p$, that is, by adding or deleting a single edge we can change $Z$ by at most $2d$, and so it follows from a standard application of the Azuma-Hoeffding inequality to the edge-exposure martingale on $Q^d_p$ (see for example \cite[Section 7]{AS}) that $Z$ is tightly concentrated about its mean, and in particular whp $Z \geq (1-\gamma_1 - c_1)n$. 
\end{proof}
It follows from Claim \ref{c:upperboundgiant} that whp 
\begin{equation}\label{e:upperboundgiant}
|V(L_1)| \leq n - Z \leq (\gamma_1 + c_1)n \leq (\gamma_1 + \delta_2)n.
\end{equation}

To continue the argument, we consider a multi-round exposure with $q_1 < p$ chosen such that if $\hat{\gamma}$ is the survival probability of the Po$(dq_1)$ branching process then $|\gamma_1 - \hat{\gamma}| \leq c_3$. Note that we then may assume that we can choose $q_2, q_3 \geq \frac{c_4}{d}$ such that 
\[(1-q_1)(1-q_2)(1-q_3) = 1-p.
\] Let us write $Q_1 = Q^d_{q_1}$, $Q_2 = Q_1 \cup Q^d_{q_2}$ and $Q_3 = Q_2 \cup Q^d_{q_3}$, where the random subgraphs are chosen independently, so that $Q_3$ has the same distribution as $Q^d_p$.

By a similar comparison to a branching process as above, we see that the probability that a fixed vertex $v$ is contained in a component of order at least $c_7 d$ in $Q_1$ is at least $\hat{\gamma} - c_6$. In particular, we expect around $\hat{\gamma} d$ of the neighbours of any fixed vertex to be contained in components of $Q_1$ of order $\Omega( d)$. Over the next few claims, we will show that we can `boost' this result to show that almost every vertex has around $\hat{\gamma} d$ neighbours contained in components of $Q_2$ of order $\Omega(d^2)$.

\begin{claim}\label{c:claim1}
Whp all but $n^{1-c_8}$ of the vertices $v$ in $Q^d$ have at least $(\hat{\gamma} - c_5)d$ many neighbours in $Q^d$ which are contained in a component of order at least $c_7 d$ in $Q_1$. Let us say in this case that $v$ has property $(a)$.
\end{claim}
\begin{proof}[Proof of Claim \ref{c:claim1}]
Indeed, each neighbour of $v$ in $Q^d$ has probability at least $\hat{\gamma} - c_6$ of being contained in a component of order at least $c_7 d$ in $Q_1$. Furthermore, since this property is an increasing property of subgraphs of $Q^d$, by the inequality of Harris (Lemma \ref{l:Harris}), given any subset $W$ of the neighbours of $v$ of size $1 \leq i \leq d$ the probability that each $w \in W$ is contained in a component of order at least $c_7 d$ in $Q_1$ is at least $(\hat{\gamma} - c_6)^i$. It follows that the number of neighbours of $v$ which are contained in a component of order at least $c_7 d$ in $Q_1$ stochastically dominates a Bin$(d,\hat{\gamma} - c_6)$ random variable. Hence, by Chernoff type bounds (Lemma \ref{l:Chernoff} \eqref{i:chernoff1}) the probability that $v$ does not have at least $d(\hat{\gamma} - c_5)$ many neighbours in $Q^d$ which are contained in a component of order at least $c_7 d$ in $Q_1$ is at most $n^{-2 c_8}$, and so $v$ does not have property $(a)$ with probability at most $n^{-2 c_8}$. Hence by Markov's inequality whp the total number of vertices without property $(a)$ is at most $n^{1-c_8}$.
\end{proof}

\begin{claim}\label{c:claim2}
Whp for all but $n^{1-c_{11}}$ of the vertices $v$ in $Q^d$ we can find a collection of at least $c_{10} d$ many disjoint connected subgraphs in $Q_1$ each of order at least $c_{12} d$ which are adjacent to $v$ in $Q^d$. Let us say in this case that $v$ has property $(b)$.
\end{claim}
\begin{proof}[Proof of Claim \ref{c:claim2}]
Indeed, consider without loss of generality the case where $v$ is the origin. We can fix $c_9 d$ disjoint subcubes of dimension $(1-c_9)d$, each of which contains a vertex adjacent to $v$, by taking for each $1 \leq i \leq c_9 d$ the vertex $w_i$ with a one in the $i$th coordinate and zeros elsewhere, and taking the subcube $Q(i)$ containing $w_i$, where we vary the last $(1 - c_9)d$ coordinates. Since $c_9 \ll \delta_1$, again by comparison to an appropriate branching process, for each $w_i$ with probability at least $\frac{\hat{\gamma}}{2}$ the component of $Q(i)_{q_1}$ containing $w_i$ has order at least $c_{12} d$. Since the subcubes $Q(i)$ are disjoint, the probability that less than $c_{10} d \leq \frac{\hat{\gamma}c_9 d}{4}$ of these components have order at least $c_{12} d$ is at most $n^{-2c_{11}}$ by Lemma \ref{l:Chernoff}. Therefore, $v$ does not have property $(b)$ with probability at most $n^{-2c_{11}}$. Hence by Markov's inequality whp the total number of vertices without property $(b)$ is at most $n^{1-c_{11}}$.
\end{proof}

\begin{claim}\label{c:claim3}
Whp all but $ n^{1-c_{13}}$ of the vertices $v$ in $Q^d$ have $(\hat{\gamma}-c_5)d$ many neighbours in $Q^d$, each of which is contained in some subgraph in $Q_1$ of order at least $c_{12} d$, where each vertex in this subgraph has property $(b)$. Let us say in this case that $v$ has property $(c)$.

\end{claim}
\begin{proof}[Proof of Claim \ref{c:claim3}]
By Claim \ref{c:claim1} whp at most $n^{1-c_8}$ vertices of $Q^d$ do not have property $(a)$. Furthermore, by Claim \ref{c:claim2} whp at most $n^{1-c_{11}}$ of the vertices of $Q^d$ do not have property $(b)$, and so whp the number of vertices at distance at most $c_{12} d$ from a vertex without property $(b)$ is at most
\[
n^{1-c_{11}} \sum_{i \leq c_{12} d} \binom{d}{i} \leq c_{12} d n^{1-c_{11}} \left(\frac{e}{c_{12}}\right)^{c_{12} d} \leq n^{1-2c_{13}}
\]
since $c_{12} \ll c_{11}$. However, clearly any vertex without property $(c)$ must either not have property $(a)$, or be within distance $c_{12} d$ from a vertex without property $(b)$. Hence, whp the number of vertices without property $(c)$ is at most $n^{1-2c_{13}} + n^{1-c_8} \leq n^{1-c_{13}}$.
\end{proof}

We now use our first sprinkling step to merge most of these disjoint connected subgraphs into ones of quadratic order. 

\begin{claim}\label{c:claim4}
Whp in $Q_2=Q_1 \cup Q^d_{q_2}$ all but $n^{1-c_{16}}$ of the vertices $v$ in $Q^d$ are adjacent in $Q^d$ to at least $(\hat{\gamma} - c_5) d$ many vertices lying in components of order at least $c_{15} d^2$ in $Q_2$. Let us say in this case that $v$ has property $(d)$.
\end{claim}
\begin{proof}[Proof of Claim \ref{c:claim4}]
After exposing $Q_1$, by Claim \ref{c:claim3} whp the number of vertices without property $(c)$ is at most $n^{1-c_{13}}$. For each vertex $v$ with property $(c)$, let us fix a family $S_v$ of $(\hat{\gamma} - c_5)d$ subgraphs of $Q_1$ of order $c_{12} d$ witnessing that $v$ has property $(c)$. Similarly for each vertex $v$ with property $(b)$, let us fix a family $T_v$ of $c_{10} d$ disjoint connected subgraphs of $Q_1$ of order $c_{12} d$ witnessing that $v$ has property $(b)$.

Consider a fixed subgraph $H \in S_v$ for some $v$ with vertex set $C = V(H)$, and let $C' \subseteq C$ be a set of vertices with the same parity of size $|C|/2$. Let $C' =\{v_1, v_2, \ldots  \}$. We will grow a vertex set $A$ containing $C$ which induces a connected subgraph of $Q_2$ inductively until it reaches an appropriate size. Let us set initially $A(0) = C$. 

If $|A(i-1)| \geq c_{15} d^2$ we let $A(i) = A(i-1)$. Otherwise, since $v$ has property $(c)$, each $v_i$ has property $(b)$, and so we can consider the family $T_{v_i}$ of disjoint subgraphs of $Q_1$ which are adjacent to $v_i$ in $Q^d$. Since $|A(i-1)| < c_{15} d^2$ we may assume that at most $\frac{c_{10}}{2}d$ of the subgraphs in $T_{v_i}$ meet $A(i-1)$ in more than $\frac{c_{12}}{2} d$ vertices. Let $T'_{v_i}$ be a subfamily of $\frac{c_{10}}{2}d$ subgraphs, each of which meets $A(i-1)$ in at most $\frac{c_{12}}{2} d$ vertices.

We expose the edges between $v_i$ and each subgraph in $T'_{v_i}$ in $Q^d_{p_2}$, and we let $A(i)$ be $A(i-1)$ together with the vertex sets of all the subgraphs in $T'_{v_i}$ which are connected to $A(i-1)$ in this manner. Let $A$ be the final subgraph built in this manner.

We say that stage $i$ is successful if $|A(i-1)| \geq c_{15} d^2$ or at least one of the exposed edges is present in $Q^d_{p_2}$. Note that the probability that a round is successful is at least
\[
1-\left(1 - \frac{c_4}{d}\right)^{\frac{c_{10}}{2}d} \geq c_{14}.
\]
Let $r$ be the number of rounds which are successful. Then $|A| \geq \min \{ r\frac{c_{12}}{2} d, c_{15} d^2 \}$.

Note that, since the parity of each $v_i \in C'$ is the same, and we only expose edges adjacent to $v_i$ in stage $i$, the outcome of each stage is independent. Hence, $r$ stochastically dominates a Bin$(|C'|,c_{14})$ random variable, whose expectation is $\frac{c_{12} c_{14}}{2} d$. Hence, it follows again from Lemma \ref{l:Chernoff} \eqref{i:chernoff1} that the probability that $|A| \leq c_{15} d^2$ is at most $n^{-4c_{16}}$.

Hence, we have shown that for any fixed subgraph $C \in S_v$, the probability that $C$ is not contained in a component of $Q_2$ of order at least $c_{15} d^2$ is at most $n^{-4c_{16}}$. It follows that the expected number of subgraphs $C$ such that $C$ is not contained in a component of $Q_2$ of order at least $c_{15} d^2$  and $C \in S_v$ for some $v$ is at most $(\hat{\gamma} - c_5)dn^{1-4c_{16}} \leq n^{1-3c_{16}}$. As before, it follows from Markov's inequality that whp the number of such subgraphs is at most $n^{1-2c_{16}}$.

In particular, whp the number of vertices with property $(c)$ but not $(d)$ is at most $n^{1-2dc_{16}}$, and so the number of vertices without property $(d)$ is at most $n^{1-c_{16}}$.
\end{proof}

In particular, since it is an increasing property and $Q_2 \subseteq Q_3$, it follows from Claim \ref{c:claim4} that property \eqref{i:badvtcs} of Theorem \ref{t:AKSfine} holds in $Q_3$, with $\delta_3 = c_{15} \ll \delta_2$, $\delta_4 =(\hat{\gamma} - c_5)$ and $\delta_5 = c_{16}$. Furthermore, since each vertex with property $(d)$ is adjacent to at least $(\hat{\gamma} - c_5)d$ many vertices contained in components of order at least $c_{15} d^2$, it follows by double counting that the number of vertices contained in components of order at least $c_{15} d^2$ in $Q_2$, and hence in $Q_3$, is at least
\begin{equation}\label{e:largevtcs}
\frac{(n-n^{1-c_{16}})(\hat{\gamma} - c_5)d}{d} \geq (\hat{\gamma} - 2c_5)n.
\end{equation}

Next, we show that whp almost all the vertices in large components in $Q_2$ belong to the largest component of $Q_3$. Hence, we will be able to deduce properties \eqref{i:largecomponent} and \eqref{i:largevtcs} of Theorem \ref{t:AKSfine} from \eqref{e:largevtcs}, together with \eqref{e:upperboundgiant}.

So, to that end, let us denote by $X$ the set of vertices without property $(d)$ and let $Y$ be the set of vertices contained in components of order at least $c_{15} d^2$ in $Q_2$. Note that, by definition, every vertex not in $X$ is adjacent to at least $(\hat{\gamma} - c_5)d$ many vertices in $Y$.

\begin{claim}\label{c:partition}
Whp for every partition of $Y = A \cup B$ with $\min \big\{ |A|,|B| \big\} \geq c_{17} n$ there is an edge between $A$ and $B$ in $Q_3 = Q_2 \cup Q^d_{q_3}$.
\end{claim}
\begin{proof}[Proof of Claim \ref{c:partition}]
Indeed, if there is a component $C$ of order at least $c_{15} d^2$ in $Q_2$ such that $C \cap A$ and $C \cap B$ are both non-empty, then there is a path between $A$ and $B$ in $Q_2 \subseteq Q_3$. Hence, we may restrict ourselves to partitions of $Y$ which partition the components of $Q_2$ of order at least $c_{15} d^2$. Since there are at most $\frac{n}{c_{15} d^2}$ many such components in $Q_2$, it follows that the total number of possible partitions to consider is at most $2^{\frac{n}{c_{15} d^2}}$. We will show that the probability that there is no path across such a partition is at most $\exp\left(-\frac{n}{c_{15} d^2}\right)$, and hence by the union bound whp the claim holds.

To this end, let us consider the inclusive neighbourhoods $N(A)$ and $N(B)$ of $A$ and $B$ respectively in $Q_2$, and let $D = N(A) \cap N(B)$. We split into two cases. Let us suppose first that $|D| \geq \frac{4 n}{c_{19} d}$.

In this case, let $D' = D \setminus X$, so that $|D'| \geq \frac{2n}{c_{19} d}$, and let $D'' \subset D'$ be a set of at least $\frac{|D'|}{2} \geq \frac{n}{c_{19} d}$ vertices of the same parity.

For each vertex $x \in D''$ we claim that the probability that there is a path from $A$ to $B$ in $Q^d_{q_3}$ of length at most two which contains $x$ is at least $\frac{c_{18}}{d}$. Indeed, suppose that $x \not\in A \cup B$, then since $x \in D$ it has at least one neighbour in both $A$ and $B$, and since $x \not\in X$ it has at least $\frac{(\hat{\gamma} - c_5)}{2} d$ neighbours in one of $A$ or $B$. Hence, such a path exists with probability at least
\[
q_3 \left(1-(1- q_3)^{\frac{(\hat{\gamma} - c_5)}{2} d}\right) \geq \frac{c_{18}}{d}.
\]
The case where $x \in A$ or $B$ is similar.

However, since these events depend on pairwise disjoint edge sets, and hence are independent for different $x$, it follows that the probability that no $x \in D''$ joins $A$ to $B$ is at most
\[
\left(1-\frac{c_{18}}{d}\right)^{|D''|} \leq \left(1-\frac{c_{18}}{d}\right)^{\frac{n}{c_{19} d}} \leq \exp\left(-\frac{n}{c_{15} d^2}\right).
\]

In the second case, we may assume that $|D| < \frac{4n}{c_{19} d}$. In this case, we note that $|N(A)| \geq |A| \geq c_{17}n$ and $|N(A)^c| \geq |B| - |D| \geq \frac{c_{17}}{2}n$. Hence, by Theorem \ref{t:vtxiso} there is a matching $F$ of size at least $\frac{2 c_{20} n}{\sqrt{d}}$ between $N(A)$ and $N(A)^c$. If we let $F' \subseteq F$ be those edges which do not have an endpoint in $D$ or $X$, then $|F'| \geq \frac{c_{20} n}{\sqrt{d}}$.

For each edge $uv=e$ in $F'$, where without loss of generality $u \in N(A) \setminus (D \cup B)$, either $u \in A$ or $u$ has at least $(\hat{\gamma} - c_5)d$ neighbours in $A$ and similarly either $v \in B$ or $v$ has at least $(\hat{\gamma} - c_5)d$ many neighbours in $B$.

Hence, we can fix as before events depending on pairwise disjoint edge sets for each $e \in F'$ which imply that there is a path between $A$ and $B$ of length at most 3 containing the edge $e$. Each of these events happens with probability at least
\[
q_3 \left(1-(1-q_3)^{(\hat{\gamma} - c_5)d}\right)^2 \geq \frac{c_{18}}{d}
\]
and so, as before, the probability that no $e \in F'$ joins $A$ to $B$ is at most
\[
\left(1-\frac{c_{18}}{d}\right)^{|F'|} \leq \left(1-\frac{c_{18}}{d}\right)^{\frac{c_{20} n}{\sqrt{d}}} \leq \exp\left(-\frac{c_{18}c_{20}n}{d^{\frac{3}{2}}}\right) \leq  \exp\left(-\frac{n}{c_{15} d^2}\right).
\]
\end{proof}

Hence, by Claim \ref{c:partition} whp there is some component $L_1$ of $Q_3$ which contains all but at most $c_{17} n$ vertices of $Y$, that is $|Y \setminus V(L_1)| \leq c_{17}n$, and hence by \eqref{e:largevtcs} 
\[
|V(L_1)| \geq (\hat{\gamma} - 2c_5 - c_{17})n \geq (\gamma_1 - c_3 - 2c_5 - c_{17})n \geq (\gamma_1 - c_1)n.
\]
However, by \eqref{e:upperboundgiant} we have $|V(L_1)| \leq (\gamma_1 + c_1 n)$, and so $|V(L_1)| = (\gamma_1 \pm c_1)n= (\gamma_1 \pm \delta_2)n$. Furthermore, by Claim \ref{c:upperboundgiant} whp at least $(1-\gamma_1 - c_1)n$ vertices of $Q^d_p$ are contained in components of order at most $d$, and hence whp every other component of $Q_3=Q^d_p$ has order at most $\max \{2c_1 n,d \} \leq \delta_2 n$. Finally, it is clear that $|V(L_1) \triangle Y| \leq 2c_1 n \leq \delta_2 n$ also, and so properties \eqref{i:largecomponent} and \eqref{i:largevtcs} of Theorem \ref{t:AKSfine} hold.
\end{proof}

\end{document}